\documentclass[a4paper,10pt,onecolumn,twoside]{article}
\usepackage{mathrsfs}
\usepackage{mathrsfs}
\usepackage{amsfonts}
\usepackage{fancyhdr}
\usepackage{amsmath,amsfonts,amssymb,graphicx,amsthm}
\allowdisplaybreaks

\usepackage{amsmath,latexsym,amsfonts,indentfirst,amsthm,amsxtra,amssymb,bm}

\usepackage{amssymb,mathrsfs,graphicx,enumerate}

\numberwithin{equation}{section}

 \newtheorem{lemma}{Lemma}[section]
 \newtheorem{proposition}[lemma]{Proposition}
 \newtheorem{theorem}{Theorem}[section]

 \theoremstyle{remark}
 \newtheorem{remark}{Remark}[section]

\numberwithin{equation}{section}

\begin{document}

\title{\bf Asymptotics of Radially Symmetric Solutions for the Exterior Problem of  Multidimensional Burgers Equation}
\author{{\bf Tong Yang}\\[1mm]
Department of Mathematics, City University of Hong Kong, China\\
Email address: matyang@cityu.edu.hk\\[2mm]
{\bf Huijiang Zhao}\\[1mm]
School of Mathematics and Statistics, Wuhan University, China\\
Computational Science Hubei Key Laboratory, Wuhan University, China\\
Email address: hhjjzhao@whu.edu.cn\\[2mm]
{\bf Qingsong Zhao}\\[1mm]
School of Mathematics and Statistics, Wuhan University, China\\
Email address: zhaoqingsong8899@live.com
}

\date{}

\maketitle

\begin{center}
{\bf Dedicated to Professor Shuxing Chen on the occasion of his 80th birthday}
\end{center}

\begin{abstract}
We are concerned with the large-time behavior of the radially symmetric solution for multidimensional Burgers equation on the exterior of a ball $\mathbb{B}_{r_0}(0)\subset \mathbb{R}^n$ for $n\geq 3$ and some positive constant $r_0>0$, where the boundary data $v_-$ and the far field state $v_+$ of the initial data are prescribed and correspond to a stationary wave. It is shown in \cite{Hashimoto-Matsumura-JDE-2019} that a sufficient condition to guarantee the existence of such a stationary wave is $v_+<0, v_-\leq |v_+|+\mu(n-1)/r_0$. Since the stationary wave is no longer monotonic, its nonlinear stability is justified only recently in \cite{Hashimoto-Matsumura-JDE-2019} for the case when $v_\pm<0, v_-\leq v_++\mu(n-1)/r_0$. The main purpose of this paper is to verify the time asymptotically nonlinear stability of such a stationary wave for the whole range of $v_\pm$ satisfying $v_+<0, v_-\leq |v_+|+\mu(n-1)/r_0$. Furthermore, we also derive the temporal convergence rate, both algebraically and exponentially. Our stability analysis is based on a space weighted energy method with a suitable chosen weight function, while for the temporal decay rates, in addition to such a space weighted energy method, we also use the space-time weighted energy method employed in \cite{Kawashima-Matsumura-CMP-1985} and \cite{Yin-Zhao-KRM-2009}.

\end{abstract}

\maketitle


\section{Introduction}

This paper is concerned with the precise description of the large time behaviors of solutions of the following initial-boundary value problem of multidimensional Burgers equation in an exterior domain $\Omega:=\mathbb{R}^n\backslash \overline{\mathbb{B}}_{r_0}(0)\subset\mathbb{R}^n$ for $n\geq 3$:
\begin{eqnarray}\label{1.1}
   {\bf u}_t+({\bf u}\cdot\nabla){\bf u}&=&\mu\Delta {\bf u}, \quad t>0,\ x\in \Omega,\nonumber\\
   {\bf u}(0,x)&=&{\bf u}_0(x),\quad x\in\Omega,\\
   {\bf u}(t,x)&=&{\bf b}(t,x),\quad t>0,\ x\in\partial\mathbb{B}_{r_0}(0),\nonumber
\end{eqnarray}
and, as in \cite{Hashimoto-NonliAnal-2014, Hashimoto-OsakaJMath-2016, Hashimoto-Matsumura-JDE-2019}, our main purpose is to understand how the space dimension $n$ effect the large time behaviors of solutions of the initial-boundary value problem \eqref{1.1}. Here ${\bf u}=\left(u_1(t,x),\cdots,u_n(t,x)\right)$ is a vector-valued unknown function of $x=(x_1,\cdots,x_n)\in \mathbb{R}^n$ and $t\geq 0,$ ${\bf u}\cdot\nabla=\sum\limits_{j=1}^nu_j\frac{\partial}{\partial x_j}$, $\mu$ and $r_0>0$ are some given positive constants. ${\bf u}_0(x)$ and ${\bf b}(t,x)$ are given initial and boundary values respectively satisfying the compatibility condition ${\bf b}(0,x)={\bf u}_0(x)$ for all $x\in\partial\mathbb{B}_{r_0}(0)$.

In this paper, we will focus on the radially symmetric solutions for the initial-boundary value problem \eqref{1.1}. In fact, if ${\bf b}(t,x)=\frac{x}{|x|}v_-, {\bf u}_0(x)=\frac{x}{|x|}v_0(|x|)$ satisfying $\lim\limits_{|x|\to+\infty}v_0(|x|)=v_+$ and $v_0(r_0)=v_-$ for some given constants $v_\pm\in\mathbb{R}$ with $v_0(|x|)$ being some given scalar function, then one can seek radially symmetric solutions to the initial-boundary value problem \eqref{1.1}. For such a case, if we introduce a new unknown function $v(t,r)$ by letting ${\bf u}(t,x)=\frac xrv(t,r)$ with $r=|x|$, then the radially symmetric solution $v(t,r):= v(t,|x|)$ of the initial-boundary value problem (\ref{1.1}) satisfies the following initial-boundary value problem
\begin{eqnarray}\label{1.2}
    v_t+\left(\frac{v^2}{2}\right)_r&=&\mu\left(v_{rr}+(n-1)\left(\frac vr\right)_r\right),\quad t>0,\ r>r_0,\nonumber\\
    v(t,r_0)&=&v_-,\quad t>0,\\
    \lim\limits_{r\rightarrow \infty }v(t,r)&=&v_+,\quad t>0,\nonumber\\
    v(0,r)&=&v_0(r),\quad r>r_0,\nonumber
\end{eqnarray}
where the initial data $v_0(r)$ is assumed to satisfy the compatibility condition
\begin{equation}\label{Compatibility condition}
v_0(r_0)=v_-,\quad \lim\limits_{r\to\infty}v_0(r)=v_+.
\end{equation}
Throughout the rest of this paper, we set $V_-=v_--\frac{\mu(n-1)}{r_0}$.

To see the influence of the space dimension $n$ on the asymptotics of the initial-boundary value problem \eqref{1.2}, one needs first to consider the corresponding problem for the one-dimensional Burgers equation in the half line
\begin{eqnarray}\label{1.1.1}
    v_t+\left(\frac{v^2}{2}\right)_r&=&\mu v_{rr},\quad r>0,\ t>0,\nonumber\\
    v(t,0)&=&v_-,\quad t>0,\\
    \lim\limits_{r\rightarrow \infty }v(t,r)&=&v_+,\quad t>0,\nonumber\\
    v(0,r)&=&v_0(r),\quad r>0,\nonumber
\end{eqnarray}
and for the initial-boundary value problem \eqref{1.1.1}, as illustrated in \cite{Matsumura-MAA-2001} and \cite{Nishihara-ADC-2001}, its large time behavior can be completely classified by the unique global entropy solution of the resulting Riemann problem of the inviscid Burgers equation
\begin{eqnarray}\label{Riemann Problem for Burgers}
v_t+\left(\frac{v^2}{2}\right)_r&=&0,\quad t>0,\ r\in\mathbb{R},\nonumber\\
v(0,r)&=&\left\{
\begin{array}{rl}
v_-,\quad &r<0,\\
v_+,\quad &r>0
\end{array}
\right.
\end{eqnarray}
together with the so-called stationary wave $\phi_{v_-,v_+}(r)$ connecting the boundary value $v_-$ and the far-field state $v_+$, which is due to the occurrence of the boundary and satisfies
\begin{eqnarray}\label{SW-1d}
\left(\frac{v^2}{2}\right)_r&=&\mu v_{rr},\quad r>0,\nonumber\\
v(0)&=&v_-,\\
\lim\limits_{r\to+\infty}v(r)&=&v_+.\nonumber
\end{eqnarray}
The rigorous mathematical justification of the above classifications for the initial-boundary value problem \eqref{1.1.1} was carried out in \cite{Liu-Matsumura-Nishihara-SIMA-1998, Liu-Nishihara-JDE-1997, Liu-Yu-ARMA-1997, Nishihara-JMAA-2001} and the results can be summarized in the Table \ref{1.1.2}. Here $\phi_{v_-,v_+}(r)$ is the stationary wave solving the problem (\ref{SW-1d}), $\psi^R_{v_-,v_+}(t,r)$ is the unique rarefaction wave solution of the Riemann problem \eqref{Riemann Problem for Burgers} for the case $v_-<v_+$, and $W_{v_-,v_+}(r-st)$ with $s=\frac{v_-+v_+}{2}$ is the suitably shifted traveling wave solution of the one-dimensional Burgers equation satisfying $W_{v_-,v_+}(\pm\infty)=v_\pm$ for the case $v_->v_+$ and $s\geq 0$.

\begin{table}\label{1.1.2}
  \centering \caption{Results available for \eqref{1.1.1}}
  \begin{tabular}{|c|c|c|}
    \hline
    \multicolumn{2}{|c|}{Boundary Condition} & Asymptotic Behavior \\
    \hline
    {$v_-<v_+$}&$v_-<v_+\leq0$ & $\phi_{v_-,v_+}$,\ \ \cite{Liu-Matsumura-Nishihara-SIMA-1998}\\
    \cline{2-3}
    &$v_-<0<v_+$ & $\phi_{v_-,0}+\psi^R_{0,v_+}$,\ \ \cite{Liu-Matsumura-Nishihara-SIMA-1998}\\
    \cline{2-3}
    &$0\leq v_-<v_+$&$\psi^R_{v_-,v_+}$,\ \ \cite{Liu-Matsumura-Nishihara-SIMA-1998}\\
    \hline
    {$v_->v_+$}&$v_-+v_+<0$&$\phi_{v_-,v_+}$,\ \ \cite{Liu-Nishihara-JDE-1997}\\
    \cline{2-3}
    &$v_-+v_+\geq0$&$W_{v_-,v_+}$,\ \ \cite{Liu-Nishihara-JDE-1997, Liu-Yu-ARMA-1997, Nishihara-JMAA-2001}\\
    \hline
  \end{tabular}
\end{table}

But for the solutions of the initial-boundary value problem (\ref{1.2}), the story is quite is different which are due to the following reasons:
\begin{itemize}
\item The first reason is that the corresponding stationary wave $\phi_{v_-,v_+}(r)$ connecting the boundary value $v_-$ and the far-field state $v_+$, which solves the following boundary problem
    \begin{eqnarray}\label{SW-RS-Burgers}
    \left(\frac{v^2}{2}\right)_r&=&\mu\left(v_{rr}+(n-1)\left(\frac vr\right)_r\right),\quad r>r_0,\nonumber\\
    v(r_0)&=&v_-,\\
    \lim\limits_{r\to+\infty}v(r)&=&v_+,\nonumber
    \end{eqnarray}
    is no longer monotonic, hence the techniques employed in \cite{Liu-Matsumura-Nishihara-SIMA-1998} and \cite{Liu-Nishihara-JDE-1997} can not be used here;
\item The second reason is that the first equation of \eqref{1.2} is nonautonomous, thus even the problem on how to construct non-stationary travelling wave $W_{v_-,v_+}(r-st)$ satisfying $W_{v_-,v_+}(\pm\infty)=v_{\pm}$ and $s\not=0$ is unknown, to say nothing about its stability.
\end{itemize}

Even so, for certain cases, it is still hopeful to use the stationary wave $\phi_{v_-,v_+}(r)$ solving \eqref{SW-RS-Burgers}, the unique rarefaction wave solution $\psi^R_{v_-,v_+}(t,r)$ of the Riemann problem \eqref{Riemann Problem for Burgers} and/or their linear superposition to describe the asymptotics of the global solutions of the initial-boundary value problem \eqref{1.2}. Such a problem was studied in \cite{Hashimoto-NonliAnal-2014, Hashimoto-OsakaJMath-2016, Hashimoto-Matsumura-JDE-2019} and the results available up to now can be summarized as follows:
\begin{itemize}
\item For the case $n=3,$ a complete classification of its asymptotic behaviors, which are similar to that of the one-dimensional Burgers equation in the half line, i.e. the initial-boundary value problem \eqref{1.1.1}, together with its rigorous mathematical justification, which includes even a linear superposition of stationary and viscous shock waves, are given in \cite{Hashimoto-Matsumura-JDE-2019};

\item For $n>3,$ only the following two cases are studied in \cite{Hashimoto-NonliAnal-2014, Hashimoto-OsakaJMath-2016, Hashimoto-Matsumura-JDE-2019}:
\begin{itemize}
\item [(a).] For the case of $V_-\leq 0\leq v_+,$ if one assume further that
\begin{equation}\label{v-+-condiiton}
0<v_-<\frac{2\mu}{r_0\left(1+\sqrt{(n-3)/(n-1)}\right)},
\end{equation}
it is shown in \cite{Hashimoto-NonliAnal-2014, Hashimoto-OsakaJMath-2016, Hashimoto-Matsumura-JDE-2019} that the asymptotics is a linear superposition of a stationary wave $\phi_{v_-,0}$ and a rarefaction wave $\psi^R_{0,v_+}.$ The nonlinear stability of the above wave pattern together with the temporal convergence rate of the global solution $v(t,r)$ of the initial-boundary value problem \eqref{1.2} toward such a wave pattern are studied in \cite{Hashimoto-NonliAnal-2014, Hashimoto-OsakaJMath-2016, Hashimoto-Matsumura-JDE-2019};

\item [(b).] For the case of $V_-\leq v_+<0$, the stationary wave $\phi_{v_-,v_+}$ is shown to be nonlinear stable, while the temporal convergence rate is unknown.
\end{itemize}
\end{itemize}

Thus for $n>3$, to the best of our knowledge, the following four cases remain unsolved:
\begin{itemize}
\item [(a).] $V_-\leq 0\leq v_+$ but $v_-\notin\left(0,\frac{2\mu}{r_0\left(1+\sqrt{(n-3)/(n-1)}\right)}\right)$;
\item [(c).] $0<V_-<v_+$;
\item [(d).] $V_->v_+, V_-+v_+<0;$
\item [(e).] $V_->v_+, V_-+v_+\geq0,$
\end{itemize}
and the main purpose of this present paper is to consider the case
\begin{equation}\label{1.3}
  v_+<0,\ \ V_-\leq |v_+|.
\end{equation}
Such a boundary condition corresponds to the stationary wave.

\begin{table}\label{1.1.3}
  \centering \caption{Results available for \eqref{1.2}}
  \begin{tabular}{|c|c|c|c|}
    \hline
    \multicolumn{2}{|c|}{Boundary Condition} & Asymptotic Behavior&Temporal convergence rate \\
    \hline
    {$V_-<v_+$}&$V_-\leq v_+<0$& $\phi_{v_-,v_+}$,\ \ \cite{Hashimoto-Matsumura-JDE-2019} \& This Paper & This Paper\\
    \cline{2-4}
    &$V_-<0\leq v_+$ & $\phi_{v_-,0}+\psi^R_{0,v_+}$,\ \  \cite{Hashimoto-NonliAnal-2014, Hashimoto-OsakaJMath-2016, Hashimoto-Matsumura-JDE-2019} for partial results&\cite{Hashimoto-Matsumura-JDE-2019} for partial results  \\
    \cline{2-4}
    &$0< V_-<v_+$&Unsolved&Unsolved\\
    \hline
    {$V_->v_+$}&$V_-\leq -v_+>0$&$\phi_{v_-,v_+}$,\ \ This Paper& This Paper\\
    \cline{2-4}
    &$V_-+v_+>0$&Unsolved&Unsolved\\
    \hline
  \end{tabular}
\end{table}

Notice that (\ref{1.3}) includes both the case (b) and the case (d). Our main results contain two parts:
\begin{itemize}
\item  For the case (b), temporal convergence rate is obtained;
\item For the case (d), the nonlinear stability of the stationary wave $\phi_{v_-,v_+}$ together with the temporal convergence rate are obtained.
\end{itemize}
The results available up to now for the rigorous mathematical justifications of the asymptotics of the initial-boundary value problem \eqref{1.2} can be summarized in the Table 2.

Before concluding this section, we outline our main ideas used in this paper. As pointed out before, this paper is concentrated on the case when the boundary condition corresponds to the stationary wave. For such a case, it is proved in Proposition 2.3 of \cite{Hashimoto-Matsumura-JDE-2019} that a sufficient condition to guarantee the unique existence of the stationary wave $\phi(r)$ to the boundary value problem \eqref{SW-RS-Burgers} is \eqref{1.3}. As for the nonlinear stability of such a stationary wave $\phi(r)$, since it is no longer monotonic, the main idea used in \cite{Hashimoto-Matsumura-JDE-2019} is to introduce the new unknown function  $z(t,r)=r^{\frac{n-1}{2}}(v(t,r)-\phi(r))$ and from \eqref{1.2} and \eqref{SW-RS-Burgers}, one can deduce that $z(t,r)$ solves (cf. (5.2) in \cite{Hashimoto-Matsumura-JDE-2019})
\begin{eqnarray}\label{Hashimoto-Matsumura-JDE-2019-formula}
z_t+(\phi z)_r+\left(-\frac{n-1}{2r}\phi+\frac{\mu (n^2-1)}{4r^2}\right)z-\mu z_{rr}&=&\frac{n-1}{2r^{\frac{n+1}{2}}}z^2-\frac{1}{2r^{\frac{n-1}{2}}}\left(z^2\right)_r,\quad t>0,\ r>r_0,\nonumber\\
z(t,r_0)&=&0,\quad t>0,\\
z(0,r)&=&z_0(r):=r^{\frac{n-1}{2}}\left(v_0(r)-\phi(r)\right),\quad r>r_0.\nonumber
\end{eqnarray}

For the initial-boundary value problem \eqref{Hashimoto-Matsumura-JDE-2019-formula}, if one performs the energy method as in \cite{Hashimoto-Matsumura-JDE-2019} to yield the zeroth-order energy type estimates, the following term appears in the left hand side of the estimates, cf. the estimate (5.2) in \cite{Hashimoto-Matsumura-JDE-2019}
\begin{equation}\label{Bad term}
I:=\frac 12\int^t_0\int^{+\infty}_{r_0}\left(\phi_r-\frac{n-1}{r}\phi +\frac{\mu(n^2-1)}{2r^2}\right)z^2drd\tau.
\end{equation}
For the one-dimensional case and if $\phi(r)$ is monotonic increasing, then such a term has a nice sign and then the nonlinear stability result follows easily. But for $n\geq 3$, since $\phi(r)$ is no longer monotonic, it is difficult to determine the sign of such a term. The main observation in \cite{Hashimoto-Matsumura-JDE-2019} is that when
\begin{equation}\label{Hashimoto-JDE-Assumptions-on-v+-}
v_\pm<0,\quad V_-\leq v_+,
\end{equation}
one can further deduce, cf. Lemma 2.4 in \cite{Hashimoto-Matsumura-JDE-2019}, that there exists a negative constant $\nu_0\leq 0$ such that
\begin{equation}\label{Upper Negative Bound on SW}
\phi(r)\leq \nu_0\leq 0,\quad r\geq r_0
\end{equation}
and it holds that $\phi(r)-\frac{\mu (n-1)^2}{2r}$ is monotonically non-decreasing for $r>r_0$, that is
\begin{equation}\label{Monotonical Property of SW}
\frac{d\phi(r)}{dr}+\frac{\mu(n-1)^2}{2r^2}\geq 0,\quad r> r_0.
\end{equation}
Having obtained the above estimates, the main idea used in \cite{Hashimoto-Matsumura-JDE-2019} is that one can deduce from the estimates \eqref{Upper Negative Bound on SW} and \eqref{Monotonical Property of SW} that $I\geq 0$ and hence the nonlinear stability result can be obtained via the elementary energy method, cf. the proofs of Theorem 3.2 in \cite{Hashimoto-Matsumura-JDE-2019} for details.

We note, however, that since our main purpose of this paper is to cover the whole range of $v_\pm$ satisfying \eqref{1.3}, which is a sufficient condition to guarantee the existence of the stationary wave $\phi(r)$, the above method developed in \cite{Hashimoto-Matsumura-JDE-2019} can not be applied and our main idea is to use the anti-derivative method to introduce the new unknown function
\begin{equation}\label{our-new-unknown}
w(t,r)=-\int_r^\infty(v(t,y)-\phi(y))dy
\end{equation}
and to use a space weighted energy method to deduce the desired nonlinear stability result. The key point in our analysis here is to introduce a suitable weight function $\chi(r)$ to overcome the difficulties induced by the non-monotonicity of the stationary wave $\phi(r)$ and the boundary condition. For details, see the properties of such a weight function $\chi(r)$ stated in Lemma \ref{3.4.1} and the proof of Theorems \ref{4.1}, \ref{4.3}, and \ref{4.5} given in Section 3. For the temporal convergence rates, both algebraically and exponentially, in addition to such a space weighted energy method, we also use the space-time weighted energy method employed in \cite{Kawashima-Matsumura-CMP-1985} and \cite{Yin-Zhao-KRM-2009}.

The rest of this paper is organized as follows. In Section 2, we first list some basic properties of the stationary wave $\phi_{v_-,v_+}(r)$, show how to construct the  weighted function $\chi(r)$ which will play an essential role in our analysis, and then state our main results. The proofs of our main results will be given in Section 3.

\vskip 2mm
\noindent \textbf{Notations:} We denote the usual Lebesgue space of square integrable functions over $(r_0,\infty)$ by $L^2=L^2((r_0,\infty))$ with norm $\|\cdot\|$ and for each non-negative integer $k$, we use $H^k$ to denote the corresponding $k$th-order Sobolev space $H^k((r_0,\infty))$ with norm $\|\cdot\|_{H^k}.$

Set $\langle r\rangle=\sqrt{1+r^2}$ and for $\alpha\in \mathbb{R},$ we denote the algebraic weighted Sobolev space, that is, the space of functions $f$ satisfying $\langle r\rangle^{\alpha/2}f\in H^k,$ by $H^{k,\alpha}$ with norm
$$
\|f\|_{k,\alpha}:=\left\|\langle r\rangle^{\alpha/2}f\right\|_{H^k}.
$$
For $k=0,$ we denote $\|\cdot\|_{0,\alpha}$ by $|\cdot|_\alpha$ for simplicity. We also denote the exponential weighted Sobolev space, that is, the space of functions $f$ satisfying $e^{\alpha r/2}f\in H^k$ for some $\alpha\in\mathbb{R}$, by $H^{k,\alpha}_{\exp}.$ For $k=0,$ we denote $\|\cdot\|^{0,\alpha}_{\exp}$ by $|\cdot|_{\alpha,\exp}$ for simplicity. For an interval $I\in \mathbb{R}^1$ and a Banach space $X,$ $C(I;X)$ denotes the space of continuous $X$-valued functions on $I,$ $C^k(I;X)$ the space of $k$-times continuously differentiable $X$-valued functions.

\section{Preliminaries and main results}

This section is devoted to the statement of our main results. Before doing so, we first collect some results obtained in \cite{Hashimoto-NonliAnal-2014, Hashimoto-OsakaJMath-2016, Hashimoto-Matsumura-JDE-2019} on the  stationary wave of (\ref{1.2}). Recall that $\phi(r)$ is called a stationary solution of (\ref{1.2}) if $\phi(r)$ solves the boundary value problem \eqref{SW-RS-Burgers}. Integrating the the first equation of \eqref{SW-RS-Burgers} with respect to $r$ from $r$ to $\infty,$ then it is easy to see that $\phi(r)$ satisfies
\begin{eqnarray}\label{2.2}
    \phi_r+\frac{(n-1)}{r}\phi&=&\frac1{2\mu}\left(\phi^2-v_+^2\right), \quad r>r_0,\nonumber\\
    \phi(r_0)&=&v_-, \\
    \lim\limits_{r\to\infty}\phi(r)&=&v_+.\nonumber
\end{eqnarray}

If we introduce a new unknown function $\psi(r)$ by
\begin{equation}\label{2.4}
  \psi(r)=\phi(r)-\frac{\mu(n-1)}{r},
\end{equation}
then \eqref{2.2} can be reformulated as
\begin{eqnarray}\label{2.3}
    \psi_r&=&\frac{1}{2\mu}\left(\psi^2-v_+^2\right)-\frac{c_0}{r^2},\quad r>r_0,\nonumber\\
    \psi(r_0)&=&V_-,\\
    \lim\limits_{r\to\infty}\psi(r)&=&v_+,\nonumber
\end{eqnarray}
where $c_0=\frac{\mu(n-1)(n-3)}{2}$ is a constant. Noticing that $c_0=0$ for $n=3,$ thus it is hopeful to expect that, for $n=3,$ the large time behavior of the unique global solution $v(t,r)$ to the initial-boundary value problem (\ref{1.2}) is similar to that of the corresponding initial-boundary value problem (\ref{1.1.1}) of the one-dimensional Burgers equation in the half line and this is the main reason why one can give a complete classification of the asymptotic behaviors of solution $v(t,r)$ to the initial-boundary value problem (\ref{1.2}), cf. \cite{Hashimoto-Matsumura-JDE-2019} for more details. For this reason, we will focus on the case $n\geq 4$ in the rest of this paper.

For the solvability of the boundary value problem \eqref{2.3} and the properties of its unique solution $\psi(r)$, we can get from Proposition 2.3 and Lemma 2.4 obtained in \cite{Hashimoto-Matsumura-JDE-2019} that

\begin{lemma}\label{2.5}
  Suppose that (\ref{1.3}) holds, then the boundary value problem \eqref{2.3} admits a unique smooth solution $\psi(r)$ satisfying
\begin{equation*}
\left|\psi(r)-v_+\right|:=\left|\phi(r)-v_+-\frac{\mu(n-1)}{r}\right|\leq O\left(r^{-2}\right),\quad r\to+\infty.
\end{equation*}
Consequently there exists a positive constant $C>0$ such that
\begin{eqnarray}\label{Properties of Stationary Wave}
    \psi(r)-v_+&\in& C^\infty\cap L^1\left(\left[r_0,\infty\right)\right),\nonumber\\
    \left|\frac{d^k\psi(r)}{d r^k}\right|&\leq& C,\quad k=1,2.
\end{eqnarray}
Moreover, if one assumes further that $v_\pm$ satisfy \eqref{Hashimoto-JDE-Assumptions-on-v+-}, then the estimates \eqref{Upper Negative Bound on SW} and \eqref{Monotonical Property of SW} hold for $r> r_0$.
\end{lemma}

Now we turn to consider the initial-boundary value problem \eqref{1.2}. If we take
\begin{equation}\label{3.2}
  w(t,r)=-\int_r^\infty(v(t,y)-\phi(y))dy,
\end{equation}
then it is easy to deduce from \eqref{1.2}, \eqref{2.4} and \eqref{2.3} that $w(t,r)$ solves the following initial-boundary value problem
\begin{eqnarray}\label{3.3}
    w_t+\psi w_r-\mu w_{rr}&=&-\frac12 w_r^2,\quad t>0,\ r>r_0,\nonumber\\
    w_r(t,r_0)&=&\lim\limits_{r\to\infty}w_r(t,r)=\lim\limits_{r\to\infty}w(t,r)=0,\quad t>0,\\
    w(0,r)&=&w_0(r):=-\int_r^\infty(v_0(y)-\phi(y))dy,\quad r>r_0.\nonumber
\end{eqnarray}

As pointed out above, compared with the initial-boundary value problem \eqref{1.1.1} for the one-dimensional Burgers equation in the half line, the main difficulty to yield the global solvability and the large time behavior of solutions to the initial-boundary value problem \eqref{3.3} is caused by the fact that the stationary wave $\phi(r)$ obtained in Lemma \ref{2.5} is no longer monotonic. Our main idea to overcome such a difficulty is to use the weighted energy method. To this end, we introduce the following weight function $\chi(r): [r_0,\infty)\rightarrow\mathbb{R}$:
\begin{equation}\label{3.4}
  \chi(r)=\exp\left(-\frac1\mu\int_{r_0}^r\psi(s)ds\right)\int_r^\infty \left(\frac2{r_0}-\frac 1s\right)\exp\left(\frac1\mu\int_{r_0}^s\psi(\tau)d\tau\right)ds.
\end{equation}
The properties of such a weight function $\chi(r)$ can be summarized in the following lemma:
\begin{lemma}\label{3.4.1}
  Suppose that (\ref{1.3}) holds, then $\chi(r)$ satisfies
\begin{itemize}
\item There exist positive constants $c_i (i=1,2)$ such that
$$
0< c_1\leq \chi(r)\leq c_2<\infty
$$
holds for each $r\geq r_0$;
\item $\chi(r)$ solves
$$
\frac{d\chi(r)}{dr}+\frac{\psi(r)}{\mu}\chi(r)=\frac 1r-\frac2{r_0},\quad r>r_0
$$
and consequently
\begin{eqnarray*}
\left.\left(\frac{d\chi(r)}{dr}+\frac{\psi(r)}{\mu}\chi(r)\right)\right|_{r=r_0}&=&-\frac {1}{r_0}<0,\\
\frac{d^2\chi(r)}{dr^2}+\frac{d}{dr}\left(\frac{\psi(r)}{\mu}\chi(r)\right)&=&-\frac 1{r^2}<0,\quad r>r_0.
\end{eqnarray*}
\end{itemize}
\end{lemma}
\begin{proof} In the following, we will prove a more general result. For each $f(r)\in C^1([r_0,\infty))$ satisfying
\begin{equation}\label{3.6}
f'(r)<0,\quad f(r_0)<0,\quad f'(r)\in L^1([r_0,\infty)).
\end{equation}
A direct corollary of the above assumptions is that the limit $\lim\limits_{r\to+\infty}|f(r)|$ exists and satisfies
\begin{equation}\label{3.6-2}
0<\lim\limits_{r\to+\infty}|f(r)|<+\infty.
\end{equation}

Now we set
\begin{equation}\label{3.5}
\widetilde{\chi}(r)=-\exp\left(-\frac1\mu\int_{r_0}^r\psi(s)ds\right)\int_r^\infty f(s)\exp\left(\frac 1\mu\int_{r_0}^s\psi(\tau)d\tau\right)ds.
\end{equation}
It is easy to see from the assumptions \eqref{3.6} and \eqref{3.6-2} imposed on $f(r)$ and the properties of $\psi(r)$ obtained in Lemma \ref{2.5} that $\widetilde{\chi}(r)$ is well-defined and satisfies
\begin{equation}\label{chi-differential equation-1}
\frac{d\widetilde{\chi}(r)}{dr}+\frac{\psi(r)}{\mu}\widetilde{\chi}(r)=f(r),\quad r>r_0.
\end{equation}
From which one can conclude from the assumptions \eqref{3.6} imposed on $f(r)$ that
\begin{eqnarray}\label{Properties of Chi}
\left.\left(\frac{d\widetilde{\chi}(r)}{dr}+\frac{\psi(r)}{\mu}\widetilde{\chi}(r)\right)\right|_{r=r_0}&=&f(r_0)<0,\nonumber\\
\frac{d^2\widetilde{\chi}(r)}{dr^2}+\frac{d}{dr}\left(\frac{\psi(r)}{\mu}\widetilde{\chi}(r)\right)&=&f'(r)<0,\quad r>r_0.
\end{eqnarray}

Now we prove that there exist positive constants $c_i (i=1,2)$ such that
\begin{equation}\label{bounds on weight function}
0< c_1\leq \widetilde{\chi}(r)\leq c_2<\infty
\end{equation}
holds for each $r\geq r_0$.

For this purpose, set
$$
A(r)=\exp\left(\frac1\mu\int_{r_0}^r\psi(s)ds\right),\quad  B(r)=\int_r^\infty|f(s)|A(s)ds,
$$
then it is easy to see that $A(r)>0, B(r)>0$ hold for $r\in[r_0,\infty)$ and consequently one can deduce that $\widetilde{\chi}(r)>0$ for each $r\geq r_0$.

Noticing that
\begin{eqnarray*}
\int^r_{r_0}\psi(s)ds&=&\int^r_{r_0}\left(\psi(s)-v_+\right)ds+v_+\left(r-r_0\right),\\
f(r)&=&f(r_0)+\int^r_{r_0}f'(s)ds,
\end{eqnarray*}
we can get from the facts that $v_+<0$, $\psi(r)-v_+\in L^1([r_0,\infty))$ and the assumptions imposed on $f(r)$ that
\begin{eqnarray}\label{3.7}
      A(r)&=&\exp\left(-\frac{|v_+|}{\mu}(r-r_0)\right) \exp\left(\frac1\mu\int_{r_0}^r(\psi(s)-v_+)ds\right)\nonumber\\
      &\leq& \exp\left(\frac1\mu\|\psi-v_+\|_{L^1}\right)
\end{eqnarray}
and
\begin{eqnarray}\label{3.8}
       B(r)&\leq& \left(\left|f\left(r_0\right)\right|+\left\|f'\right\|_{L^1}\right) \exp\left(\frac1\mu\|\psi-v_+\|_{L^1}\right)\int_r^\infty \exp\left(-\frac{|v_+|}{\mu}(s-r_0)\right)ds\nonumber\\
       &=&\frac{\mu\left(\left|f(r_0)\right|+\|f'\|_{L^1}\right)}{|v_+|} \exp\left(\frac1\mu\|\psi-v_+\|_{L^1}-\frac{|v_+|}\mu(r-r_0)\right)\\
       &\leq& \frac{\mu\left(\left|f(r_0)\right|+\|f'\|_{L^1}\right)} {|v_+|} \exp\left(\frac1\mu\|\psi-v_+\|_{L^1}\right)\nonumber
\end{eqnarray}
hold for each $r\geq r_0$. Moreover, the estimates \eqref{3.7} and \eqref{3.8} implies that
\begin{equation}\label{3.9}
    \lim\limits_{r\to+\infty} A(r)=\lim\limits_{r\to+\infty}B(r)=0.
\end{equation}

From \eqref{3.9} and l'Hopital's rule, one can deduce from the assumption \eqref{3.6-2} imposed on $f(r)$ that
\begin{equation}\label{3.10}
  \lim\limits_{r\to\infty}\widetilde{\chi}(r)=\lim\limits_{r\to\infty}\frac{B(r)}{A(r)} =\lim\limits_{r\to\infty}\frac{B'(r)}{A'(r)} =\frac\mu{|v_+|}\lim\limits_{r\to\infty}|f(r)|\in (0,+\infty).
\end{equation}

Having obtained \eqref{3.7}, \eqref{3.8} and \eqref{3.10}, the estimate \eqref{bounds on weight function} follows immediately.

If we take $f(r)=\frac1r-\frac2{r_0}$, then one can easily deduce the results stated in Lemma \ref{3.4.1}.
\end{proof}

With the above preparations in hand, we now turn to state our main results. The first result is concerned with the nonlinear stability of the stationary wave $\phi(r)$ constructed in Lemma \ref{2.5}.
\begin{theorem}\label{4.1}
  Suppose that (\ref{1.3}) holds and $w_0\in H^2.$ Then, there exists a sufficiently small positive constant $\epsilon$ such that if $\|w_0\|_{H^2}\leq\epsilon,$ the initial-boundary value problem (\ref{1.2}) admits a unique global solution $v(t,r)$ satisfying
  \begin{equation*}
    v(t,r)-\phi(r)\in C(0,\infty;H^1)\cap L^2_{loc}(0,\infty;H^2)
  \end{equation*}
  and
  \begin{equation}\label{4.2}
    \lim\limits_{t\to +\infty}\sup_{r\geq r_0}|v(t,r)-\phi(r)|=0.
  \end{equation}
\end{theorem}

For the temporal convergence rates of the unique global solution $v(t,r)$ of the initial-boundary value problem (\ref{1.2}) toward the stationary wave $\phi(r)$, we first have the following result on the algebraic decay rates.
\begin{theorem}\label{4.3}
  Under the assumptions imposed in Theorem \ref{4.1}, then for any $\alpha>0,$ if we assume further that $w_0\in L^2_\alpha,$  we can get that
  \begin{equation}\label{4.4}
    \sup_{r\geq r_0}|v(t,r)-\phi(r)|\leq C\left(\|w_0\|_{H^2}+|w_0|_\alpha\right)(1+t)^{-\frac \alpha 2}
  \end{equation}
holds for $t\geq 0$. Here $C$ is some time-independent positive constant $C>0$.
\end{theorem}

For the temporal exponential decay estimates, we have
\begin{theorem}\label{4.5}
  Under the assumptions listed in Theorem \ref{4.1}, for any $\beta$ and $\gamma$ satisfies
  \begin{equation}\label{4.6}
      0< \beta\leq \min\left\{\frac2{r_0},\frac{8}{(8\|\chi\|_{L^\infty}+1)r_0}\right\},\quad 0<\gamma\leq\frac{3\mu\beta}{8r_0\|\chi\|_{L^\infty}},
  \end{equation}
if we assume further that $w_0\in H^{2,\beta}_{\exp},$ then there exists a time-independent positive constant $C>0$ such that
   \begin{equation}
     \sup_{r\geq r_0}|v(t,r)-\phi(r)|\leq C\|w_0\|_{\exp}^{2,\beta}\exp(-\gamma t)
   \end{equation}
holds for all $t\geq 0$.
\end{theorem}

\begin{remark} In our main results, we ask the $H^2-$norm of the initial perturbation $w_0(r)$ to be sufficiently small, it would be an interesting problem to deduce the corresponding nonlinear stability result for large initial perturbation and such a problem is under our current research. For related results on the nonlinear stability of stationary waves for certain hyperbolic conservation laws with dissipation and large initial perturbation, those interested are referred to \cite{Fan-Liu-Wang-Zhao-JDE-2014, Fan-Liu-Zhao-AA-2013, Fan-Liu-Zhao-JHDE-2011, Fan-Liu-Zhao-Zou-KRM-2013} and the references cited therein.
\end{remark}

\section{Proof of our main results}

This section is devoted to the proofs of our main results. To make the presentation clear, we divide this section into three subsections, the first one focuses on the proof of  Theorem \ref{4.1}.
\subsection{Proof of Theorem \ref{4.1}}

To prove Theorem \ref{4.1}, for some positive constants $T>0$ and $M>0$, we first define the set of functions for which we seek the solution of the initial-boundary value problem (\ref{3.3}) by
\begin{equation*}
  X_M(0,T)=\left\{w(t,r)\in C\left([0,T];H^2\right), w_r(t,r)\in L^2\left([0,T];H^2\right),\ \sup\limits_{t\in [0,T]}\|w(t)\|_{H^2}\leq M\right\}.
\end{equation*}
Put
\begin{equation*}
  N(T)=\sup_{0\leq t\leq T}\|w(t)\|_{H^2},
\end{equation*}
then we can get the following local existence result.
\begin{proposition}[Local Existence]\label{5.1}
  Under the assumptions listed in Theorem \ref{4.1}, for any positive constant $M,$ there exists a positive constant $T=T(M)$ such that if $\|w_0\|_{H^2}\leq M,$ the initial-boundary value problem (\ref{3.3}) has a unique solution $w(t,r)\in X_{2M}(0,T).$
\end{proposition}

Proposition \ref{5.1} can be proved by a standard iterative method, so we omit the details for brevity.

Our second result is concerned with the following a priori estimates on the local solution $w(t,r)$ constructed in Proposition \ref{5.1}, from which and the local solvability result Proposition \ref{5.1}, Theorem \ref{4.1} follows immediately.

\begin{proposition}[A Priori Estimates]\label{5.2}
  Under the assumptions listed in Theorem \ref{4.1}, suppose that $w(t,r)\in X_M(0,T)$ is a solution of (\ref{3.3}) defined on the time interval $[0,T].$ Then if one assumes further that there exists a sufficiently small positive constant $\epsilon_1>0$, which is independent of $M$ and $T$, such that if $N(T)\leq \epsilon_1,$ the following estimate holds
  \begin{equation}\label{5.2.0}
    \|w(t)\|_{H^2}^2+\int_0^t\left(\left\|\frac{w(\tau)}{r}\right\|^2+\|w_r(\tau)\|_{H^2}^2 +w^2(\tau,r_0)\right)d\tau\leq C\|w_0\|_{H^2}^2,\ \ t\in[0,T]
  \end{equation}
  for and $t\in[0,T]$ and some positive constant $C$ independent of $M$ and $T$.
\end{proposition}
\begin{remark} In fact by choosing the weight function $\widetilde{\chi}(r)$ suitably, we can obtain better a priori estimates. For example, if we take $f(r)=\frac{1}{2\epsilon}\left(r^{-2\epsilon}-r_0^{-2\epsilon}\right)-1$ in $\widetilde{\chi}(r)$ for $\epsilon>0$ in (\ref{3.5}), then we can deduce the following new a priori estimates
  \begin{equation*}
    \|w(t)\|_{H^2}^2+\int_0^t\left(\left\|\frac{w(\tau)}{r^{\frac12+\epsilon}}\right\|^2 +\left\|w_r(\tau)\right\|_{H^2}^2+w^2(\tau,r_0)\right)d\tau\leq C\|w_0\|_{H^2}^2,\ \ t\in[0,T].
  \end{equation*}
  However, $\epsilon$ can not take $0,$ which is different from the a priori estimates obtained in \cite{Hashimoto-Matsumura-JDE-2019}.
\end{remark}
\begin{proof}
   Firstly we deduce the zeroth-order energy type estimates. To this end, we first notice that Lemma \ref{3.4.1} means that the norms $\|\cdot\|_\chi$ and $\|\cdot\|$ are equivalent. Then
   multiplying $(\ref{3.3})_1$ by $\chi w$, we can get that
   \begin{equation*}
     \left(\frac 12\chi w^2\right)_t-\frac\mu2\left(\chi_{rr}+\frac1\mu(\psi\chi)_r\right)w^2 +\mu\chi w_r^2+\mu\left(-\chi ww_r+\frac 12\left(\chi_r+\frac\psi\mu\chi\right)w^2\right)_r=-\frac 12\chi ww_r^2.
   \end{equation*}

Lemma \ref{3.4.1} together with the above identity yields
   \begin{equation}\label{5.2.1}
     \left(\frac 12\chi w^2\right)_t+\frac\mu2\frac{w^2}{r^2}+\mu\chi w_r^2 -\mu\left(\chi ww_r+\left(\frac1{r_0}-\frac1{2r}\right)w^2\right)_r=-\frac 12\chi ww_r^2.
   \end{equation}

Integrating the above identity with respect to $r$ over $(r_0,\infty),$ and noticing that
  \begin{equation*}
    -\frac 12\int_{r_0}^\infty\chi ww^2_rdr\leq C\|w\|_{L^\infty}\|w_r\|^2
    \leq CN(T)\|w_r\|^2,
  \end{equation*}
we obtain
  \begin{equation*}
    \frac 12\frac {d}{dt}\|w(t)\|_\chi^2 +\frac{\mu}2\left\|\frac{w(t)}{r}\right\|^2+\mu\|w_r(t)\|_\chi^2+\frac \mu{2r_0}w^2(t,r_0)\leq CN(T)\|w_r\|^2.
  \end{equation*}

Integrating the above equation with respect to $t$ over $[0,T],$ we then have the zeroth-order estimate as long as $N(T)$ is assumed to be small enough
  \begin{equation}\label{5.3}
    \|w(t)\|^2+\mu\int_0^t\left(\left\|\frac{w(\tau)}{r}\right\|^2+\|w_r(\tau)\|^2+\frac 1{r_0}w^2(\tau,r_0)\right)d\tau\leq C\|w_0\|^2.
  \end{equation}

Now we turn to deal with the first-order energy type estimates. For this purpose, differentiating $(\ref{3.3})_1$ with respect to $r$ once and multiplying the resulting identity by $w_r$ and integrating the resulting equation with respect to $r$ over $[r_0,\infty),$ we obtain
  \begin{equation}\label{5.4}
     \frac 12\frac{d}{dt}\|w_r(t)\|^2+\mu\left\|w_{rr}(t)\right\|^2\leq C \int_{r_0}^\infty \left(\left|\psi_rw_r^2\right|+\left|\psi w_rw_{rr}\right|+\left|w_r^2w_{rr}\right| \right)(t,r) dr.
  \end{equation}

From the properties of the stationary wave $\psi(r)$ obtained in Lemma \ref{2.5}, the right hand side of the above inequality can be estimates as follows:
  \begin{eqnarray*}
    \int_{r_0}^\infty\left|\psi_r(r)w^2_r(t,r)\right|dt&\leq&C\left\|w_r(t)\right\|^2,\\
    \int_{r_0}^\infty \left|\psi(r) w_r(t,r)w_{rr}(t,r)\right|dr&\leq& C \left\|w_r(t)\right\|^2+\frac\mu4\left\|w_{rr}(t)\right\|^2,\\
    \int_{r_0}^\infty\left|w_r^2(t,r)w_{rr}(t,r)\right|dr&\leq& \left\|w_r(t)\right\|_{L^\infty}\int_{r_0}^\infty\left|w_r(t,r)w_{rr}(t,r)\right|dr\\
    &\leq& CN(T)\left(\left\|w_r(t)\right\|^2+\left\|w_{rr}(t)\right\|^2\right).
  \end{eqnarray*}
Substituting the above estimates into \eqref{5.4} and by choosing $N(T)$ sufficiently small, one can get from the zeroth-order energy type estiamtes (\ref{5.3}) that
  \begin{equation}\label{5.5}
    \left\|w_r(t)\right\|^2+\mu\int_0^t\left\|w_{rr}(\tau)\right\|^2d\tau\leq C\left\|w_0\right\|_{H^1}^2.
  \end{equation}

Finally, we deduce the second-order energy type estimates. To do so, differentiating $(\ref{3.3})_1$ with respect to $r$ twice yields
  \begin{equation}\label{5.6}
    w_{rrt}-\mu w_{rrrr}=-w_{rr}^2-w_rw_{rrr}-\psi_{rr}w_r-2\psi_rw_{rr}-\psi w_{rrr}.
  \end{equation}
Multiplying (\ref{5.6}) by $w_{rr}$ and integrating the resulting equation with respect to $r$ over $[r_0,\infty),$ we get that
  \begin{eqnarray}\label{5.7}
    &&\frac 12\frac d{dt}\left\|w_{rr}(t)\right\|^2+\mu\left\|w_{rrr}(t)\right\|^2\\
    &\leq& -\mu w_{rr}(t,r_0)w_{rrr}(t,r_0)
    -\int_{r_0}^\infty\left( w_{rr}^3+w_rw_{rr}w_{rrr} +\psi_{rr}w_rw_{rr}+2\psi_rw^2_{rr}+\psi w_{rr}w_{rrr} \right)(t,r)dr.\nonumber
  \end{eqnarray}
 Now we turn to estimate the terms in the right hand side of \eqref{5.7}. For the first term, if we take $r=r_0$ in (\ref{5.6}), then we can get from the facts $w_r(t,r_0)=0, \psi(r_0)=V_-$ that
  \begin{equation*}
    w_{rrr}(t,r_0)=\frac{V_-}{\mu}w_{rr}(t,r_0).
  \end{equation*}
Consequently, from the above identity and the Sobolev inequality, the first term in the right hand side of \eqref{5.7} can be estimated as
  \begin{eqnarray}\label{5.9}
      \left|w_{rr}(t,r_0)w_{rrr}(t,r_0)\right|&\leq& C\left|w_{rr}(t,r_0)\right|^2\nonumber\\
       &\leq& C\left\|w_{rr}(t)\right\|\left\|w_{rrr}(t)\right\|\\
      &\leq& \frac{\mu}4\left\|w_{rrr}(t)\right\|^2+C\left\|w_{rr}(t)\right\|^2.\nonumber
  \end{eqnarray}

For the second term in the right hand side of (\ref{5.7}), the Sobolev inequality and the definition of $N(T)$ tell us that
  \begin{eqnarray}\label{5.10}
      -\int_{r_0}^\infty \left(w_{rr}^3+w_rw_{rr}w_{rrr}\right)(t,r)dr
      &=&\int_{r_0}^\infty w_r(t,r)w_{rr}(t,r)w_{rrr}(t,r)dr\nonumber\\
      &\leq& \left\|w_r(t)\right\|_{L^\infty}\int_{r_0}^\infty\left|w_{rr}(t,r)w_{rrr}(t,r)\right|dr\\
      &\leq& CN(T)\left(\left\|w_{rr}(t)\right\|^2+\left\|w_{rrr}(t)\right\|^2\right),\nonumber
  \end{eqnarray}
  \begin{equation}\label{5.11}
    -\int_{r_0}^\infty\psi_{rr}(r)w_r(t,r)w_{rr}(t,r)dr\leq C\left(\left\|w_r(t)\right\|^2+\left\|w_{rr}(t)\right\|^2\right),
  \end{equation}
  \begin{equation}\label{5.12}
    -2\int_{r_0}^\infty\psi_r(r)  w^2_{rr}(t,r)dr\leq C\left\|w_{rr}(t)\right\|^2,
  \end{equation}
  \begin{equation}\label{5.13}
     -\int_{r_0}^\infty \psi(r) w_{rr}(t,r)w_{rrr}(t,r)dr\leq \frac\mu4\left\|w_{rrr}(t)\right\|^2+C\left\|w_{rr}(t)\right\|^2.
  \end{equation}

Substituting the estimates (\ref{5.9})-(\ref{5.13}) into (\ref{5.7}), integrating the resulting inequality with respect to $t$ from $0$ to $t,$
and by making use of the zeroth-order energy type estimates (\ref{5.3}) and the first-order energy type estimates (\ref{5.5}), we have
  \begin{equation}\label{5.14}
  \left\|w_{rr}(t)\right\|^2+\mu\int_0^t\left\|w_{rrr}(\tau)\right\|^2d\tau\leq C\left\|w_0\right\|_{H^2}.
  \end{equation}

Having obtained the estimates (\ref{5.3}), (\ref{5.5}) and (\ref{5.14}), the a priori estimate (\ref{5.2.0}) follows immediately and this completes the proof of Proposition \ref{5.2}.
\end{proof}

\subsection{Proof of Theorem \ref{4.3}}

We prove Theorem \ref{4.3} in this subsection. To this end, the key point is to deduce the following result.
\begin{lemma}\label{5.15}
  Under the assumptions listed in Theorem \ref{4.3}, for any $0\leq\beta\leq \alpha$ and $\gamma\geq0,$ it holds that
  \begin{eqnarray}\label{5.16}
    &&(1+t)^\gamma|w(t)|_\beta^2+\int_0^t\left\{\beta(1+\tau)^\gamma|w(\tau)|^2_{\beta-1} +(1+\tau)^\gamma\left|w_r(\tau)\right|^2_\beta
    +(1+\tau)^\gamma w^2(\tau,r_0)\right\}d\tau\nonumber\\
    &\leq& C\left\{\left|w_0\right|_\beta^2 +\gamma\int_0^t(1+\tau)^{\gamma-1}|w(\tau)|_\beta^2d\tau+
    \beta\int_0^t(1+\tau)^\gamma\left\|w_r(\tau)\right\|^2d\tau\right\}.
  \end{eqnarray}
\end{lemma}
\begin{proof}
  Multiplying (\ref{5.2.1}) by $\langle r\rangle^\beta$ and integrating the resultant equation with respect to $r$ over $[r_0,\infty),$ we obtain
  \begin{eqnarray}\label{5.24}
      &&\left(\frac12\int_{r_0}^\infty\chi(r)\langle r\rangle^\beta w^2(t,r)dr\right)_t+\mu\int_{r_0}^\infty\left\{A_\beta(r)\langle r\rangle^{\beta-1}w^2(t,r)+\chi(r)\langle r\rangle^\beta w_r^2(t,r)\right\}dr+\frac\mu{2r_0}\langle r_0\rangle^\beta w^2(t,r_0)\nonumber\\
      &\leq& C\beta\int_{r_0}^\infty\langle r\rangle^{\beta-1}\left|w(t,r)w_r(t,r)\right|dr,
  \end{eqnarray}
  where $A_\beta(r)=\frac{\langle r\rangle}{2r^2}+\beta\frac{r}{\langle r\rangle}\left(\frac1{r_0}-\frac1{2r}\right)\geq \frac{\beta}{2\langle r_0\rangle}.$

The term in the right hand side of (\ref{5.24}) can controlled by
   \begin{eqnarray}\label{algebraic-decay}
       &&C\beta\int_{r_0}^\infty\langle r\rangle^{\beta-1}\left|w(t,r)w_r(t,r)\right|dr\nonumber\\
       &\leq&\frac{\mu\beta}{4\langle r_0\rangle} \int_{r_0}^\infty\langle r\rangle^{\beta-1}w^2(t,r)dr+\beta C\left(\int_{r_0}^R+\int_{R}^\infty\right)\langle r\rangle^{\beta-1}w_r^2(t,r)dr\\
       &\leq&\frac{\mu\beta}{4\langle r_0\rangle}\int_{r_0}^\infty\langle r\rangle^{\beta-1}w^2(t,r)dr +\beta C(R)\int_{r_0}^\infty w_r^2(t,r)dr+\frac{\beta C}R\int_{r_0}^\infty\langle r\rangle^\beta w_r^2(t,r)dr,\nonumber
   \end{eqnarray}
where $C(R)$ depends only on $R$.

By choosing $R$ suitable large in \eqref{algebraic-decay}, we can get from \eqref{5.24}, \eqref{algebraic-decay} and Lemma \ref{3.4.1} that
   \begin{equation}\label{add-1}
     \left(\int_{r_0}^\infty\chi(r)\langle r\rangle^\beta w^2(t,r)dr\right)_t+\int_{r_0}^\infty\left(\langle r\rangle^{\beta-1}w^2(t,r) +\langle r\rangle^\beta w_r^2(t,r)\right)dr+w^2(t,r_0) \leq C\beta\left\|w_r(t)\right\|^2.
   \end{equation}

Having obtained \eqref{add-1}, the estimate \eqref{5.16} follows immediately by multiplying \eqref{add-1} by $(1+t)^\gamma$, integrating the resultant equation with respect to $t$ from $0$ to $t$ and the properties on the weight function $\chi(r)$ obtained in Lemma \ref{3.4.1}. This completes the proof of Lemma \ref{5.15}.
\end{proof}

Having obtained Lemma \ref{5.15}, by repeating the arguments used in \cite{Kawashima-Matsumura-CMP-1985}, we can deduce the following two results which is crucial to yield the temporal convergence rate for $\|w(t)\|.$
\begin{lemma}\label{5.17}
  It holds that for $k=0,1,\cdots,[\alpha],$
  \begin{equation}\label{5.18}
    (1+t)^k|w(t)|^2_{\alpha-k}+\int_0^t\left\{(\alpha-k)(1+\tau)^k|w(\tau)|^2_{\alpha-k-1} +(1+\tau)^k\left|w_r(\tau)\right|^2_{\alpha-k}
    +(1+\tau)^kw^2(\tau,r_0)\right\}d\tau\leq C|w_0|^2_\alpha.
  \end{equation}
\end{lemma}

\begin{lemma}\label{5.19}
  It holds for any $\epsilon>0$ that
  \begin{equation}\label{5.20}
    (1+t)^{\alpha+\epsilon}\|w(t)\|^2+\int_0^t(1+\tau)^{\alpha+\epsilon} \left\{\left\|w_r(\tau)\right\|^2+
    w^2(\tau,r_0)\right\}d\tau\leq C(1+t)^\epsilon|w_0|_\alpha^2.
  \end{equation}
\end{lemma}

To deduce the corresponding temporal convergence rates on the high-order derivatives of the solution $w(t,r)$ with respect to $r$, we can get the following lemma.

\begin{lemma}\label{5.21}
  Let $l=1,2,$ then it holds for any $\epsilon>0$ and $t\geq 0$ that
  \begin{equation}\label{5.22}
    (1+t)^{\alpha+\epsilon}\left\|\partial_r^lw(t)\right\|^2 +\int_0^t(1+\tau)^{\alpha+\epsilon}\left\|\partial_r^lw_r(\tau)\right\|^2
    d\tau \leq C(1+t)^\epsilon\left(\left\|w_0\right\|_{H^2}^2+\left|w_0\right|^2_\alpha\right).
  \end{equation}
\end{lemma}

With the above results in hand, we now turn to prove Theorem \ref{4.3}. To do so, we first consider the convergence rate of $\|w(t)\|.$ If $\alpha$ is an integer, then we can take $k=\alpha$ in Lemma \ref{5.17} to deduce that
  \begin{equation}
    (1+t)^\alpha\|w(t)\|^2+\int_0^t(1+\tau)^\alpha\left\{\left\|w_r(\tau)\right\|^2 +w_0^2(\tau,r_0)\right\}d\tau\leq C|w_0|_\alpha^2,
  \end{equation}
  which means that
  \begin{equation}\label{5.23}
    \|w(t)\|\leq C (1+t)^{-\frac\alpha2}|w_0|_\alpha,
  \end{equation}
while if $\alpha$ is not an integer, we can use Lemma \ref{5.19} to yield (\ref{5.23}).

For the corresponding temporal convergence rates on $\|\partial_r^lw(t)\|$ for $l=1,2$, we can deduce from Lemma \ref{5.21} that
  \begin{equation}
    \|w(t)\|_{H^2}\leq C(1+t)^{-\frac\alpha2}\left(\|w_0\|_{H^2}+|w_0|_\alpha\right).
  \end{equation}
Consequently, the Sobolev inequality gives
  \begin{eqnarray}
      \sup\limits_{r\geq r_0}|v(t,r)-\phi(r)|&=&\sup\limits_{r\geq r_0}\left|w_r(t,r)\right|\nonumber\\
      &\leq& C\left\|w_r(t)\right\|^{\frac12}\left\|w_{rr}(t)\right\|^{\frac12}\\
      &\leq& C\left(\|w_0\|_{H^2}+|w_0|_\alpha\right)(1+t)^{-\frac\alpha 2}.\nonumber
  \end{eqnarray}
This is \eqref{4.4} and the proof of Theorem \ref{4.3} is completed.

\subsection{Proof of Theorem \ref{4.5}}

To prove Theorem \ref{4.5}, we first give the following result.

\begin{lemma}\label{5.3.0}
  Suppose that (\ref{4.6}) holds and let $B_\beta(r)=\frac{1}{2r^2}+\beta\left(\frac1{r_0}-\frac1{2r}\right),$ then it holds for $r\geq r_0$ that
  \begin{equation}\label{5.3.1}
    B_\beta(r)\geq\max\left\{\frac{3\beta}{4r_0}, \beta^2\chi(r)\right\}.
  \end{equation}
\end{lemma}
\begin{proof} Since
$$
\frac{dB_\beta(r)}{dr} =\frac{\beta r-2}{2r^3},
$$
if $\beta\leq \frac{2}{r_0},$ one can deduce that
$$
\min\limits_{r\geq r_0}B_\beta(r)=B_\beta\left(\frac{2}{\beta}\right),
$$
and consequently for $r\geq r_0$
\begin{equation}\label{add-exponential-1}
B_\beta(r)\geq B_\beta\left(\frac{2}{\beta}\right)=\frac{\beta(8-\beta r_0)}{8r_0}\geq \frac{3\beta}{4r_0}.
\end{equation}
While if $\beta\leq \frac{8}{(8\|\chi\|_{L^\infty}+1)r_0},$ one can deduce that \begin{equation}\label{add-exponential-2}
B_\beta(r)-\beta^2\chi(r)\geq\frac{\beta^2}{r_0}\left(\frac1\beta-\frac{\left(8\|\chi\|_{L^\infty}+1\right) r_0}{8}\right)\geq 0.
\end{equation}

\eqref{add-exponential-1} together with \eqref{add-exponential-2} imply \eqref{5.3.1}, this completes the proof of Lemma \ref{5.3.0}.
\end{proof}

For the exponential temporal convergence rates on $w(t,r)$, we have

\begin{lemma}\label{5.3.2}
  Under the assumptions listed in Theorem \ref{4.5}, it holds for all $t\geq 0$ that
  \begin{equation}\label{5.3.3}
    e^{\gamma t}|w(t)|_{\beta,\exp}^2 +\int_0^te^{\gamma\tau}\left\{\beta|w(\tau)|^2_{\beta,\exp} +\left|w_r(\tau)\right|^2_{\beta,\exp}+w^2(\tau,r_0) \right\}d\tau\leq C|w_0|^2_{\beta,\exp}.
  \end{equation}
\end{lemma}
\begin{proof}
  Multiplying (\ref{5.2.1}) by $e^{\beta r}$ and then integrating the result with respect to $r$ from $r_0$ to $\infty$ yield
  \begin{eqnarray}\label{add-exponential-3}
      &&\left(\frac12\int_{r_0}^\infty e^{\beta r}\chi(r) w^2(t,r)dr\right)_t+\mu \int_{r_0}^\infty e^{\beta r}\left\{B_\beta(r)w^2(t,r)+\chi(r) w_r^2(t,r)\right\}dr+\frac{\mu}{2r_0} e^{\beta r_0}w^2(t,r_0)\nonumber\\
      &=&\underbrace{-\frac 12\int_{r_0}^\infty e^{\beta r}\chi(r) w(t,r)w_r^2(t,r)dr}_{I_1}\underbrace{-\mu\beta\int_{r_0}^\infty e^{\beta r}\chi(r) w(t,r)w_r(t,r)dr}_{I_2}.
  \end{eqnarray}
Now we turn to estimate $I_j (j=1,2)$ term by term. Firstly, $I_1$ is estimated as
  \begin{equation*}
      \left|I_1\right|\leq C \|w(t)\|_{L^\infty}\left|w_r(t)\right|^2_{\beta,\exp}
      \leq CN(T)\left|w_r(t)\right|^2_{\beta,\exp},
  \end{equation*}
while for $I_2$, we can get that
  \begin{eqnarray*}
      \left|I_2\right|&\leq&\frac{\mu}2\int_{r_0}^\infty e^{\beta r}\cdot 2\cdot \beta\chi(r)^{\frac12}|w(t,r)|\cdot\chi(r)^{\frac12}\left|w_r(t,r)\right|dr\\
      &\leq&\frac{\mu}2\int_{r_0}^\infty e^{\beta r}\left\{\beta^2\chi(r)|w(t,r)|^2 +\chi(r) w_r(t,r)^2\right\}dr\\
      &\leq&\frac{\mu}2\int_{r_0}^\infty e^{\beta r}\left\{B_\beta(r)w^2(t,r)+\chi(r) w_r(t,r)^2\right\}dr,
  \end{eqnarray*}
  where we have used Lemma \ref{5.3.0}.

Substituting the above estimates into \eqref{add-exponential-3} and by choosing $N(T)$ sufficiently small such that $N(T)\leq \frac{\mu}{4C}$, we can derive that
  \begin{equation}\label{5.3.4}
    \left(\int_{r_0}^\infty e^{\beta r}\chi(r) w^2(t,r)dr\right)_t +\frac{3\beta\mu}{4r_0}|w|^2_{\beta,\exp}+\frac{\mu}2\int_{r_0}^\infty e^{\beta r}\chi(r) w_r(t,r)^2dr+\frac{\mu e^{\beta r_0}}{r_0}w^2(t,r_0)\leq 0.
  \end{equation}

Multiplying (\ref{5.3.4}) by $e^{\gamma t}$ and integrating the result with respect to $t$ from $0$ to $t,$ we obtain
  \begin{eqnarray}\label{5.3.41}
      &&e^{\gamma t}\int_{r_0}^\infty e^{\beta r}\chi(r) w^2(t,r)dr+\int_0^te^{\gamma \tau}\left\{\frac{3\beta\mu}{4r_0}|w(\tau)|^2_{\beta,\exp}+\frac\mu2\int_{r_0}^\infty e^{\beta r}\chi(r) w_r(\tau,r)^2dr+\frac{\mu e^{\beta r_0}}{r_0}w^2(t,r_0)\right\}d\tau\nonumber\\
      &\leq& C|w_0|^2_{\beta,\exp}+
      \gamma\int_0^t e^{\gamma \tau}\int_{r_0}^\infty e^{\beta r}\chi(r) w^2(t,r)drd\tau.
  \end{eqnarray}

From the assumption (\ref{4.6}), the last term in the right hand side of (\ref{5.3.41}) can be bounded by
  \begin{equation}\label{add-exponential-5}
      \gamma\int_0^t e^{\gamma \tau}\int_{r_0}^\infty e^{\beta r}\chi(r) w^2(t,r)drd\tau\leq\gamma\|\chi\|_{L^\infty}
      \int_0^t e^{\gamma\tau}|w(\tau)|^2_{\beta,\exp}d\tau
      \leq \frac{3\beta\mu}{8r_0}\int_0^t e^{\gamma\tau}|w(\tau)|^2_{\beta,\exp}d\tau.
  \end{equation}

Putting \eqref{5.3.41} and \eqref{add-exponential-5} together, one can deduce (\ref{5.3.3}) immediately. This completes the proof of Lemma \ref{5.3.2}.
\end{proof}

Similarly, for the corresponding estimates on $\partial^l_rw(t,r)$ for $l=1,2$, we can get that

\begin{lemma}\label{5.3.5}
  Under the assumptions listed in Theorem \ref{4.5}, it holds for $t\geq 0$ that
  \begin{equation}\label{5.3.6}
    e^{\gamma t}\left|w_r(t)\right|_{\beta,\exp}^2 +\int_0^te^{\gamma\tau}\left|w_{rr}(\tau)\right|^2_{\beta,\exp}d\tau\leq C\|w_0\|^2_{H^{1,\beta}_{\exp}}
  \end{equation}
  and
  \begin{equation}\label{5.3.7}
    e^{\gamma t}\left|w_{rr}(t)\right|_{\beta,\exp}^2 +\int_0^te^{\gamma\tau}\left|w_{rrr}(\tau)\right|^2_{\beta,\exp}d\tau\leq C|w_0|^2_{H^{2,\beta}_{\exp}}.
  \end{equation}
\end{lemma}

Having obtained Lemma \ref{5.3.2} and Lemma \ref{5.3.5}, Theorem \ref{4.5} follows easily and we omit the details for brevity.

\section*{Acknowledgements} The work was supported by the Fundamental Research Funds for the Central Universities and two grants from the National Natural Science Foundation of China under contracts 11731008 and 11671309, respectively.


\begin{thebibliography}{99}\small

\bibitem{Fan-Liu-Wang-Zhao-JDE-2014}L.-L Fan, H.-X. Liu, T. Wang, and H.-J. Zhao, Inflow problem for the one-dimensional compressible Navier-Stokes equations under large initial perturbation. {\it J. Differential Equations} {\bf 257} (2014), no. 10, 3521-3553.

\bibitem{Fan-Liu-Zhao-AA-2013} L.-L. Fan, H.-X. Liu, and H.-J. Zhao, One-dimensional damped wave equation with large initial perturbation. {\it Anal. Appl. (Singap.)} {\bf 11} (2013), no. 4, 1350013, 40 pp.

\bibitem{Fan-Liu-Zhao-JHDE-2011} L.-L. Fan, H.-X. Liu, and H.-J. Zhao, Nonlinear stability of planar boundary layer solutions for damped wave equation. {\it J. Hyperbolic Differ. Equ.} {\bf 8} (2011), no. 3, 545-590.

\bibitem{Fan-Liu-Zhao-Zou-KRM-2013} L.-L. Fan, H.-X. Liu, H.-J. Zhao, and Q.-Y. Zou, Global stability of stationary waves for damped wave equations. {\it Kinet. Relat. Models} {\bf 6} (2013), no. 4, 729-760.

\bibitem{Hashimoto-NonliAnal-2014} I. Hashimoto, Asymptotic behavior of radially symmetric solutions for Burgers equation in several space dimensions. {\it Nonlinear Anal.} {\bf 100} (2014), 43-58.

\bibitem{Hashimoto-OsakaJMath-2016} I. Hashimoto, Behavior of solutions for radially symmetric solutions for Burgers equation with a boundary corre-sponding to the rarefaction wave. {\it Osaka J. Math.} {\bf 53} (2016), 799-811.

\bibitem{Hashimoto-Matsumura-JDE-2019} I. Hashimoto, A. Matsumura, Asymptotic behavior toward nonlinear waves for radially symmetric solutions of the multi-dimensional Burgers equation. {\it J. Differential Equations} {\bf 266} (2019), 2805-2829.

\bibitem{Kawashima-Matsumura-CMP-1985} S. Kawashima and A. Matsumura, Asymptotic stability of traveling wave solutions of systems for one-dimensional gas motion. {\it Commun. Math. Phys.} {\bf 101} (1985), 97-127.

\bibitem{Liu-Matsumura-Nishihara-SIMA-1998} T.-P. Liu, A. Matsumura, and K. Nishihara, Behaviors of solutions for the Burgers equation with boundary corresponding to rarefaction waves. {\it SIAM J. Math. Anal.} {\bf 29} (1998), 293-308.

\bibitem{Liu-Nishihara-JDE-1997} T.-P. Liu and K. Nishihara, Asymptotic behavior for scalar viscous conservation laws with boundary effect. {\it J. Differential Equations} {\bf 133} (1997), 296-320.

\bibitem{Liu-Yu-ARMA-1997}T.-P. Liu and S.-H. Yu, Propagation of a stationary shock layer in the presence of a boundary. {\it Arch. Rational Mech. Anal.} {\bf 139} (1997), no. 1, 57-82.

\bibitem{Matsumura-MAA-2001}A. Matsumura, Inflow and outflow problems in the half space for a one-dimensional isentropic model system of compressible viscous gas. IMS Conference on Differential Equations from Mechanics (Hong Kong, 1999). {\it Methods Appl. Anal.} {\bf 8} (2001), no. 4, 645-666.


\bibitem{Nishihara-JMAA-2001} K. Nishihara, Boundary effect on a stationary viscous shock wave for scalar viscous conservation laws. {\it J. Math. Anal. Appl.} {\bf 255} (2001), no. 2, 535-550.

\bibitem{Nishihara-ADC-2001} K. Nishihara, Asymptotic behaviors of solutions to viscous conservation laws via $L^2-$energy method. {\it Adv. Math. (China)} {\bf 30} (2001), no. 4, 293-321.

\bibitem{Yin-Zhao-KRM-2009} H. Yin and H.-J. Zhao, Nonlinear stability of boundary layer solutions for generalized Benjamin-Bona-Mahony-Burgers equation in the half space. {\it Kinetic and Ralated Models} {\bf 2} (2009), 521-550.
\end{thebibliography}
\end{document}